\documentclass{amsart}

\usepackage[utf8]{inputenc}
\usepackage{amsmath}
\usepackage{amssymb}
\usepackage{amsthm}
\usepackage{hyperref}

\usepackage[dvipdfmx]{graphicx}
\usepackage{tikz}

\usepackage{hyperref}
\usepackage{enumerate}
\theoremstyle{plain}
\newtheorem{thm}{Theorem}
\newtheorem{lem}[thm]{Lemma}

\theoremstyle{definition}

\theoremstyle{remark}

\colorlet{grayred}{yellow!30!brown}
\colorlet{graygreen}{yellow!70!brown}
\colorlet{darkgreen}{black!45!green}

\renewcommand{\ge}{\geqslant}

\title{An aperiodic subtraction game of Nim-dimension two}

\author{Urban Larsson}
\address{Urban Larsson, Department of Mathematics and Statistics, Dalhousie University, 6316 Coburg Road, PO Box 15000, Halifax, Nova Scotia, Canada B3H 4R2, supported by the Killam Trust}
\email{urban031@gmail.com}
\date{March 18, 2015}
\begin{document}

\maketitle
\begin{abstract}
In a recent arXiv-manuscript Fox studies infinite subtraction games with a finite (ternary) and aperiodic Sprague-Grundy function. Here we provide an elementary example of a game with the given properties, namely the game given by the subtraction set $\{F_{2n+1}-1\}$, where $F_i$ is the $i$th Fibonacci number, and where $n$ ranges over the positive integers. Our definition of nim-dimension reflects the precise number of power-of-two-components generated by the games; the group of nim-values is of order four so the dimension is two (in the classical definition this dimension would have been one). Thanks to Carlos Santos for an enlightening discussion on this matter.
\end{abstract}

\section*{A simple solution to a problem about aperiodic subtraction games suggested by Nathan Fox}

In a recent arXiv-manuscript \cite{F} Nathan Fox studies infinite and aperiodic subtraction games \cite{B} with a finite (ternary) Sprague-Grundy function \cite{S,G}. In this note we provide an elementary example of a game with the given properties. In particular, this means our game has nim-dimension two\footnote{It is the number of power-of-two-components that defines the group of nim-values generated by the games; this group is of order four so the dimension is two. In the classical definition \cite{Sa} this dimension would have been one.}.
 Let $\phi=\frac{1+\sqrt{5}}{2}$ denote the Golden ratio. Let $S=\{F_{2n+1}-1\}=\{1,4,12,\ldots \}$, where $n$ ranges over the positive integers, and where $F_i$ is the $i$th Fibonacci number.
\begin{lem}\label{lem1}
The sets $B=\{\lfloor n\phi^2 \rfloor \}_{n\geqslant 1}$, $B+1=\{\lfloor n\phi^2 \rfloor +1 \}_{n\geqslant 0}$, and $AB + 1=\{2\lfloor n\phi \rfloor + n+1 \}_{n\geqslant 1}$ partition the positive integers. 
\end{lem}

\begin{proof} 
By \cite{W}, it suffices to prove that the sets $B+1$ and $AB+1$ partition the set $A=\{\lfloor n\phi \rfloor\}_{n\geqslant 1}$.  It is well known that $x\in A$ if and only if the smallest Fibonacci term in the Zeckendorf expansion of $x$ has an even index e.g. \cite{S}. Let $z_i$ denote the $i$th smallest index of a Fibonacci term in the representation of numbers in the respective sequences. For $AB+1$, $z_1 = 2$ (hence in $A$) since it has all representatives with $z_2\geqslant 4$ even. We know that $B+1\subset A$, contains all representatives with $z_1\geqslant 4$ even, since $B$ contains all representatives with $z_1=3$ (since $F_4=F_3+1$, $F_6= F_5+F_3+1$ and so on). It also contains all representatives with $z_2\geqslant 5$ odd and $z_1=2$, and it also contains the representative with just $z_1$.
\end{proof}

Note that because the Golden ratio is an irrational number, the sets in Lemma~\ref{lem1} are aperiodic (in fact they follow a beautiful fractal pattern \cite{L} related to the Fibonacci morphism).
The two-player subtraction game $S$ is played as follows. The players alternate in moving. From a given nonnegative integer $x$, the current player moves to a new integer of the form $x-s\ge 0$, where $s\in S$. A player unable to move, because no number in $s$ satisfies the inequality, loses. 

The Sprague-Grundy value of an impartial game is computed recursively as the least nonnegative integer not in the set of values of the move options, and starting with the terminal position(s) which have Sprague-Grundy value zero.

\begin{thm} 
The Sprague-Grundy value of the subtraction game $S$ is $g(x)=0$ if $x\in B, g(x)=1$ if $x\in B+1$ and $g(x)=2$ if $x\in AB+1$. 
\end{thm}

\begin{proof}
The base case is $g(0)=0$. Suppose that the result holds for all $m<n$. We begin by showing that, if $b\in B$, then no follower of $b$ is in $B$. It suffices to show that $b-F_{2i+1}+1\in A$, which is true if and only if the Zeckendorf representation's smallest term is even indexed. It holds trivially unless $b-F_{2i+1}$ contains the smallest term $F_3$ or $F_2$. If it contains the former, then we compute $F_3+F_2$ and get $F_4$. Unless $F_5$ is contained we are done. Continuing this argument gives the claim in the first case. Suppose, for a contradiction, that $F_2$ is the smallest Zeckendorf term of $b-F_{2i+1}$. If the smallest Zeckendorf term, say $F_{2j+1}$ in $b$ has index greater than or equal to $2i+1$, then $F_2$ is not the smallest term. Hence $j<i$. 

Suppose next that $x\in B+1$. We need to show that there is a follower in $B$, but no follower in $B+1$. Let $b=x-1$. Then $b+1-(F_{2i+1}-1)=b-F_{2i+1}+F_3\in B$ if $i=1$ (which solves the first part). Suppose now, that $x$ has a follower in $B+1$. Then $b+1-(F_{2i+1}-1)\in B+1$, which contradicts an argument in the previous paragraph. 

At last we prove that if $x\in AB+1$ then $x$ has both a follower in $B$ and in $B+1$, but no follower in $AB+1$. We begin with the latter. We want to show that $x-F_{2i+1}+1 \not\in AB+1$. Thus it suffices to show $\alpha = ab+1-F_{2i+1}\not\in AB$ (where $ab = x-1$). The only way to not having the least representative as $F_2=1$ (recall $z_1\geqslant 4$ for numbers in $AB$) is to have $j = 1$ and, for the number $ab$, $z_1 = 4$. But, for $\alpha$, this gives $z_1 = 3$ and hence $\alpha\not\in AB$. Next, we find an $i$ such that $ab+1-(F_{2i+1}-1)\in B+1$, that is such that $ab-F_{2i+1}+1\in B$. Take $i=1$. It suffices to show that $ab-1\in B$. But $F_{2k}-F_2=F_3+\cdots +F_{2k-1}$, for any $k>1$ (with only odd indexes in the sum) and we may assume that $ab$ has least index $2k$. It remains to find an $i$ such that $ab+1-(F_{2i+1}-1)\in B$, that is such that $ab+F_3-F_{2i+1}\in B$. Suppose again that $2k$ is the least index of a Zeckendorf representative in $ab$. Then $k>1$ and, if $k>2$, we may take $i=k-1$ to obtain $F_{2k-2}$ as the least representative in the number $(ab-F_{2i+1})$. If $2k-2\geqslant 4$ we are done. At last, assume that $k=2$. Then the smallest representative in the number $ab+F_3$ is odd indexed say $2j+1\geqslant 5$. We can take $i=j+1$, unless $ab+F_3=F_{2j+1}$, in which case we take $i=j$ and the option is 0.
\end{proof}

\end{document}